\documentclass[10pt]{amsart}
\usepackage[style=alphabetic]{biblatex} 
\usepackage{tikz}
\usepackage{amssymb}
\usepackage{amsthm}
\usepackage[all]{xy}
\usepackage{microtype}
\usepackage{xcolor}
\usepackage{hyperref}
\usepackage[nameinlink]{cleveref}
\usepackage{graphicx}

\makeatletter
    
    \@addtoreset{equation}{section}
  \makeatother

\hypersetup{
 colorlinks,
 linkcolor={teal},
 citecolor={teal},
 urlcolor={teal}
}

\calclayout

\DeclareFieldFormat
  [article,book,inbook,incollection,inproceedings,patent,thesis,unpublished]
  {title}{\emph{#1\isdot}}

\theoremstyle{plain}
	\newtheorem{theorem}{Theorem}[section]
	\newtheorem{lemma}[theorem]{Lemma}
	\newtheorem{corollary}[theorem]{Corollary}

        \newtheorem{conjecture}[theorem]{Conjecture}

\theoremstyle{definition} 
	\newtheorem{remark}[theorem]{Remark}
	\newtheorem{definition}[theorem]{Definition}
	\newtheorem{example}[theorem]{Example}

\begin{document}
\title{On iterated circumcenter sequences}
\author[S. Kanda]{Shuho Kanda}
\address{Graduate School of Mathematical Sciences, University of Tokyo, 3-8-1 Komaba, Meguro-ku, Tokyo 153-8914, Japan}
\email{shuho@ms.u-tokyo.ac.jp}
\author[J. Koizumi]{Junnosuke Koizumi}
\address{RIKEN iTHEMS, Wako, Saitama 351-0198, Japan}
\email{junnosuke.koizumi@riken.jp}

\date{\today}
\thanks{}

\begin{abstract}
    An iterated circumcenter sequence (ICS) in dimension $d$ is a sequence of points in $\mathbb{R}^d$ where each point is the circumcenter of the preceding $d+1$ points.
    The purpose of this paper is to completely determine the parameter space of ICSs and its subspace consisting of periodic ICSs.
    In particular, we prove Goddyn's conjecture on periodic ICSs, which was independently proven recently by Ardanuy.
    We also prove the existence of a periodic ICS in any dimension.
\end{abstract}

\maketitle
\setcounter{tocdepth}{1}
\tableofcontents

\enlargethispage*{20pt}
\thispagestyle{empty}

\section{Introduction}

This paper is motivated by the following conjecture posted on the website \emph{Open Problem Garden} \cite{OPG_circumcenter}, proposed by Luis Goddyn:

\begin{conjecture}[Goddyn]\label{conj}
    Let $d\geq 2$ be an integer.
    Let $p=(p_i)_{i=1}^\infty$ be a sequence of points in $\mathbb{R}^d$ with the property that for every $i \geq 1$, the points $p_i,p_{i+1},\dots,p_{i+d}$ are distinct, lie on a unique sphere, and $p_{i+d+1}$ is the center of this sphere.
    If this sequence is periodic, then its period must be $2d+4$.
\end{conjecture}

We shall refer to a sequence $p$ that satisfies the conditions in the conjecture as a $d$-dimensional \emph{iterated circumcenter sequence} (ICS).
In this paper, we completely determine the parameter space of ICSs and its subspace consisting of periodic ICSs, which in particular gives an affirmative answer to Conjecture \ref{conj}.
We note that Conjecture \ref{conj} was independently proven recently by Ardanuy \cite{Ardanuy}.

Our strategy is as follows.
We say that an ICS $p$ is \emph{special} if $|p_1-p_i|=|p_2-p_i|=\dots=|p_{i-1}-p_i|$ holds for $3\leq i\leq d+1$.
Since any ICS becomes special after some shift, we may focus on special ICSs.
For each special ICS, we define its \emph{characteristic sequence} $(a_i)_{i=1}^\infty$ by
$$
    a_i=\dfrac{|p_i-p_{i+1}|^2}{4|p_{i+1}-p_{i+2}|^2}.
$$
It is easy to see that the characteristic sequence determines a special ICS up to similarity transformations.
Therefore our tasks are the following:
\begin{enumerate}
    \item Parametrize all possible characteristic sequences.
    \item Analyze the structure of each ICS in terms of its characteristic sequence.
\end{enumerate}

As for (1), we prove that the characteristic sequence satisfies the following recurrence relation:
$$
    a_{i+d-1}=1-\dfrac{a_{i+d-2}}{1-\dfrac{a_{i+d-3}}{1-\dfrac{a_{i+d-4}}{\ddots-\dfrac{a_{i+1}}{1-a_i}}}}\quad (i\geq 1).
$$
Somewhat surprisingly, this coincides with the recurrence relation for the so-called \emph{Lyness cycles} \cite{Lyness_cross_ratio}; see also \cite{Griffiths2012}.
In particular, it follows that the characteristic sequence is always $(d+2)$-periodic.
Using this result, we can parametrize characteristic sequences by a certain open subset of $\mathbb{R}^{d-1}$ defined by polynomial inequalities:

\begin{theorem}
    For $n\geq 1$, we define a polynomial $F^{(n)}(x_1,x_2,\dots,x_n)$ by
    $$
        F^{(n)}(x_1,x_2,\dots,x_n)=\sum_{\substack{A\subset \{1,2,\dots,n\}\\ \forall i,j\in A,\:|i-j|\geq 2}} (-1)^{|A|}x_A,
    $$
    where $x_A=\prod_{i\in A}x_i$.
    Then the characteristic sequences of special $d$-dimensional ICSs are parametrized by
    $$
        U_d = \{(x_1,x_2,\dots,x_{d-1})\in \mathbb{R}^{d-1}\mid x_i>0,\;F^{(i)}(x_1,x_2,\dots,x_i)>0\;(1\leq i\leq d-1)\}.
    $$
\end{theorem}

As for (2), we prove the following theorem, which describes the structure of a general ICS:
\begin{theorem}
    Let $p$ be a special $d$-dimensional ICS with characteristic sequence $(a_i)_{i=1}^\infty$ and set
    $$r=\dfrac{1}{2^{d+2}\sqrt{a_1a_2\cdots a_{d+2}}}.$$
    Then there exists a point $v \in \mathbb{R}^d$ such that $p_{i+d+2} =  v-r p_i$ holds for $i\geq 1$.
    In particular, $p$ is periodic if and only if $r=1$.
\end{theorem}
Using this general result, Conjecture \ref{conj} can be proven by a simple argument.
We also show that $r$ can take any value in $[\cos^{d+2}(\pi/(d+2)),\infty)$, which in particular implies that a periodic ICS exists in any dimension.

\subsection*{Acknowledgement}
A large part of this paper is based on discussions at ``Mathspace Topos''.
The authors would like to express their deep gratitude to Nobuo Kawakami, the promoter of ``Mathspace Topos'', Fumiharu Kato, the adviser, and Toshihiko Nakazawa, the organizer.
The authors would like to thank Ryuya Hora for pointing out that Lemma \ref{key_lemma} (3) can be proved by induction.
The authors would also like to thank Masaki Natori and Koto Imai for their contributions to the visualization of ICSs.

\section{Characteristic sequences}

Throughout the paper, we assume that $d\geq 2$ is an integer.
For $p_1,p_2,\dots,p_n\in \mathbb{R}^d$, we write $H(p_1,p_2,\dots,p_n)$ for the affine subspace spanned by $p_1,p_2,\dots,p_n$.
We say that $p_1,p_2,\dots,p_n$ are \emph{in general position} if $H(p_1,p_2,\dots,p_n)$ has dimension $n-1$.
In this case there is a unique $(n-2)$-dimensional sphere in $H(p_1,p_2,\dots,p_n)$ containing $p_1,p_2,\dots,p_n$.
We write $S(p_1,p_2,\dots,p_n)$ for this sphere.

\begin{definition}
    We say that $p_1,p_2,\dots,p_n\in \mathbb{R}^d$ are \emph{in good position} if they are in general position and satisfy $|p_1-p_i|=|p_2-p_i|=\dots=|p_{i-1}-p_i|$ for $3\leq i\leq n$.
\end{definition}

\begin{figure}[h]
\begin{tikzpicture}[scale=0.8]
\coordinate (P1) at (0.3,0);
\coordinate (P2) at (3,0);
\coordinate (P3) at (4,1.5);
\coordinate (P4) at (2,4);

\draw (P1) -- (P2) -- (P3);
\draw [dashed] (P1) -- (P3);
\draw (P1) -- (P4);
\draw (P2) -- (P4);
\draw (P3) -- (P4);

\draw[dashed] (P1) -- (P3);

\path (P2) -- node[sloped, midway] {\tiny $||$} (P3);
\path (P1) -- node[sloped, midway] {\tiny $||$} (P3);

\path (P1) -- node[sloped, midway] {\tiny $|$} (P4);
\path (P2) -- node[sloped, midway] {\tiny $|$} (P4);
\path (P3) -- node[sloped, midway] {\tiny $|$} (P4);

\node[below left] at (P1) {$p_1$};
\node[below right] at (P2) {$p_2$};
\node[above right] at (P3) {$p_3$};
\node[above] at (P4) {$p_4$};

\end{tikzpicture}
\caption{Four points $p_1,p_2,p_3,p_4$ in good position}
\end{figure}
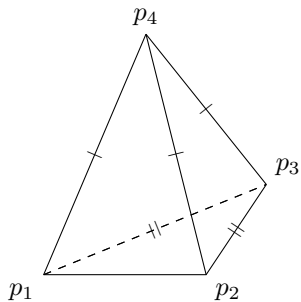

\begin{definition}
    A \emph{$d$-dimensional iterated circumcenter sequence} (ICS) is an infinite sequence $p=(p_i)_{i=1}^\infty$ of points in $\mathbb{R}^d$ such that for any $i\geq 1$, the points $p_i,p_{i+1},\dots,p_{i+d}$ are in general position, and $p_{i+d+1}$ is the center of $S(p_i,p_{i+1},\dots,p_{i+d})$.
    
    We say that $p$ is \emph{special} if $p_1,p_2,\dots,p_{d+1}$ are in good position.
    In this case, $p_i,p_{i+1},\dots,p_{i+d}$ are in good position for all $i\geq 1$.
    We say that $p$ is \emph{periodic} if there is an integer $m>0$ such that $p_i=p_{i+m}$ holds for $i\geq 1$.
    The smallest value of such $m$ is called the \emph{period} of $p$.
\end{definition}

\begin{remark}\label{remark_shift}
    If $p$ is a $d$-dimensional ICS, then the shifted sequence $(p_{i+d-1})_{i=1}^\infty$ is a special $d$-dimensional ICS.
    In particular, a periodic ICS is special.
\end{remark}

The following figure illustrates an example of a periodic $2$-dimensional ICS with period $8$.

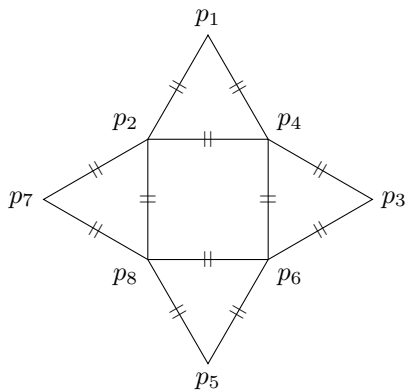
\begin{figure}[h]
\begin{tikzpicture}[scale = 0.8]
\def\sidelength{2}
\coordinate (P1) at (0, {(1+sqrt(3))/2*\sidelength});
\coordinate (P2) at (-{\sidelength/2}, {\sidelength/2});
\coordinate (P3) at ({(1+sqrt(3))/2*\sidelength}, 0);
\coordinate (P4) at ({\sidelength/2}, {\sidelength/2});
\coordinate (P5) at (0, -{(1+sqrt(3))/2*\sidelength});
\coordinate (P6) at ({\sidelength/2}, -{\sidelength/2});
\coordinate (P7) at (-{(1+sqrt(3))/2*\sidelength}, 0);
\coordinate (P8) at (-{\sidelength/2}, -{\sidelength/2});
\draw (P2) -- node[sloped, midway] {\tiny $||$}
(P4) -- node[sloped, midway] {\tiny $||$}
(P6) -- node[sloped, midway] {\tiny $||$}
(P8) -- node[sloped, midway] {\tiny $||$}
cycle;
\draw (P1) -- node[sloped, midway] {\tiny $||$}
(P4) -- node[sloped, midway] {\tiny $||$}
(P3) -- node[sloped, midway] {\tiny $||$}
(P6) -- node[sloped, midway] {\tiny $||$}
(P5) -- node[sloped, midway] {\tiny $||$}
(P8) -- node[sloped, midway] {\tiny $||$}
(P7) -- node[sloped, midway] {\tiny $||$}
(P2) -- node[sloped, midway] {\tiny $||$}
cycle;
\node at (P1) [above] {$p_1$};
\node at (P2) [above left] {$p_2$};
\node at (P3) [right] {$p_3$};
\node at (P4) [above right] {$p_4$};
\node at (P5) [below] {$p_5$};
\node at (P6) [below right] {$p_6$};
\node at (P7) [left] {$p_7$};
\node at (P8) [below left] {$p_8$};
\end{tikzpicture}
\caption{A periodic $2$-dimensional ICS with period $8$}
\end{figure}

\begin{lemma}\label{key_lemma}
    Let $n\geq 2$ and suppose that $p_1,p_2,\dots,p_n\in \mathbb{R}^d$ are in good position.
    For $1\leq i\leq n$, we write $Q_i$ (resp. $R_i$) be the center (resp. radius) of $S(p_1,p_2,\dots,p_i)$.
    \begin{enumerate}
        \item   The half-line $p_nQ_{n-1}$ is perpendicular to $H(p_1,p_2,\dots,p_{n-1})$, and $Q_n$ lies on this half-line.
        \item   If we set $b_i=|p_i-p_{i+1}|$ and $a_i=b_i^2/4b_{i+1}^2$, then we have
                $$
                R_n^2 =
                    \dfrac{b_{n-1}^2/4}{
                    1-\dfrac{a_{n-2}}{
                        1-\dfrac{a_{n-3}}{
                            \ddots -\dfrac{a_2}{1-a_1}
                        }
                    }
                }.
                $$
        \item   The segment $p_1Q_n$ is disjoint from $H(p_2,p_3,\dots,p_n)$.
                In particular, $p_2,p_3,\dots,p_n,Q_n$ are in good position.
    \end{enumerate}
\end{lemma}

\begin{proof}
    Let $H=H(p_1,p_2,\dots,p_{n-1})$.
    First we show that the line $p_nQ_{n-1}$ is perpendicular to $H$.
    Suppose that $1\leq i<j\leq n-1$.
    Since $p_1,p_2,\dots,p_n$ are in good position, we have
    $|p_i-p_n|=|p_j-p_n|$.
    On the other hand, we have
    $|p_i-Q_{n-1}|=|p_j-Q_{n-1}|$
    by definition of $Q_{n-1}$.
    Therefore the line $p_nQ_{n-1}$ is perpendicular to the line $p_ip_j$, which proves the claim.
    The same argument shows that the line $Q_{n-1}Q_n$ is perpendicular to $H$, so the points $p_n,Q_{n-1},Q_n$ are collinear.
    The following figure illustrates two possibilities for the configurations of the four points  $p_{n-1},p_n,Q_{n-1},Q_n$.
    In either case, $|p_{n-1}-Q_n|=|p_n-Q_n|=R_n$ shows that $Q_n$ lies on the half-line $p_nQ_{n-1}$.
    This completes the proof of (1).

    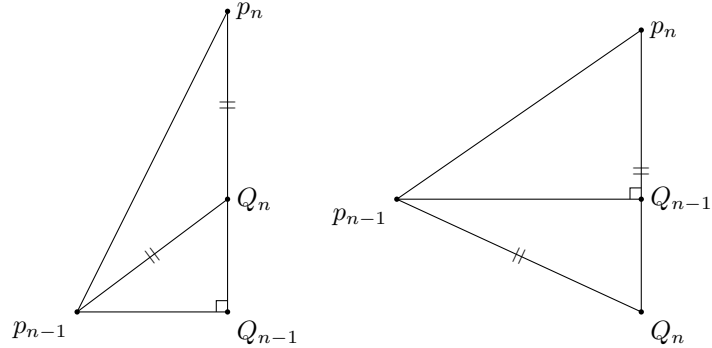
\begin{figure}[h]
    \begin{tikzpicture}[scale = 0.5]
    \coordinate (A) at (0,0);
    \coordinate (B) at (4,0);
    \coordinate (C) at (4,3);
    \coordinate (D) at (4,8);

    \draw (A) -- (B);
    \draw (A) -- (C) node[sloped, midway] {\tiny $||$};
    \draw (A) -- (D);
    \draw (B) -- (C);
    \draw (C) -- (D) node[sloped, midway] {\tiny $||$};
    \draw (3.7,0) -- (3.7,0.3) -- (4,0.3);

    \foreach \point in {A, B, C, D}
     \fill (\point) circle (2pt);

    \node at (A) [below left] {$p_{n-1}$};
    \node at (B) [below right] {$Q_{n-1}$};
    \node at (C) [right] {$Q_n$};
    \node at (D) [right] {$p_n$};
    

    \coordinate (A) at (8.5,3);
    \coordinate (B) at (15,3);
    \coordinate (C) at (15,0);
    \coordinate (D) at (15,7.5);

    \draw (A) -- (B);
    \draw (A) -- (C) node[sloped, midway] {\tiny $||$};
    \draw (A) -- (D);
    \draw (C) -- (D) node[sloped, midway] {\tiny $||$};
    \draw (14.7,3) -- (14.7,3.3) -- (15,3.3);

    \foreach \point in {A, B, C, D}
     \fill (\point) circle (2pt);

    \node at (A) [below left] {$p_{n-1}$};
    \node at (B) [right] {$Q_{n-1}$};
    \node at (C) [below right] {$Q_n$};
    \node at (D) [right] {$p_n$};
    \end{tikzpicture}
    \caption{Two possible configurations of points $p_{n-1},p_n,Q_{n-1},Q_n$}
    \end{figure}

    Next we prove (2).
    By the above figure and the Pythagorean theorem, we get
    $$
        R_n^2 = \dfrac{b_{n-1}^2/4}{1-\biggl(\dfrac{R_{n-1}}{b_{n-1}}\biggr)^2}.
    $$
    Applying this to the first $i$ points $p_1,p_2,\dots,p_i$, we also get
    $$
        \biggl(\dfrac{R_i}{b_i}\biggr)^2 = \dfrac{a_{i-1}}{1-\biggl(\dfrac{R_{i-1}}{b_{i-1}}\biggr)^2}\quad (2\leq i\leq n-1).
    $$
    The equality (2) follows from these formulas.

    We prove (3) by induction on $n$.
    For $n=2$, the claim is obvious.
    Suppose that $n\geq 3$.
    Applying the induction hypothesis to the points $p_1,p_2,\dots,p_{n-1}\in H(p_1,p_2,\dots,p_{n-1})$, we see that the segment $p_1Q_{n-1}$ is disjoint from $H(p_2,p_3,\dots,p_{n-1})$.
    Noting that $p_n\not\in H(p_1,p_2,\dots,p_{n-1})$, this implies that the segment $p_1Q_{n-1}$ is disjoint from $H(p_2,p_3,\dots,p_n)$.
    On the other hand, since $Q_n$ lies on the half-line $p_nQ_{n-1}$ by (1), the segment $Q_{n-1}Q_n$ is also disjoint from $H(p_2,p_3,\dots,p_n)$.
    Combining these two facts, we conclude that the segment $p_1Q_n$ is disjoint from $H(p_2,p_3,\dots,p_n)$.
\end{proof}

\begin{corollary}\label{extension_lemma}
    Suppose that $p_1,p_2,\dots,p_{d+1}\in \mathbb{R}^d$ are in good position.
    Then we can take $p_{d+2},p_{d+3},\dots\in \mathbb{R}^d$ so that $(p_i)_{i=1}^\infty$ is a special $d$-dimensional ICS.
\end{corollary}

\begin{proof}
    We define $p_{d+2}$ to be the center of $S(p_1,p_2,\dots,p_{d+1})$.
    Then Lemma \ref{key_lemma} (3) shows that $p_2,p_3,\dots,p_{d+2}$ are also in good position.
    Repeating this process, we get the desired sequence.
\end{proof}

\begin{definition}
    Let $p$ be a $d$-dimensional ICS.
    We define the \emph{characteristic sequence} of $p$ to be the sequence of positive real numbers $(a_i)_{i=1}^\infty$ defined by
    $$
        a_i = \dfrac{|p_i-p_{i+1}|^2}{4|p_{i+1}-p_{i+2}|^2}.
    $$
\end{definition}

\begin{lemma}\label{recurrence}
    Let $p$ be a special $d$-dimensional ICS.
    Then the characteristic sequence $(a_i)_{i=1}^\infty$ satisfies
    \begin{align}
        \label{continued_fraction_ineq}
        a_i&<1-\dfrac{a_{i-1}}{1-\dfrac{a_{i-2}}{1-\dfrac{a_{i-3}}{\ddots-\dfrac{a_{2}}{1-a_1}}}}\quad (1\leq i\leq d-1),\\
        \label{continued_fraction}
        a_{i+d-1}&=1-\dfrac{a_{i+d-2}}{1-\dfrac{a_{i+d-3}}{1-\dfrac{a_{i+d-4}}{\ddots-\dfrac{a_{i+1}}{1-a_i}}}}\quad (i\geq 1).
    \end{align}
    Conversely, for any sequence of positive real numbers $(a_1,a_2,\dots,a_{d-1})$ satisfying (\ref{continued_fraction_ineq}), there is a special $d$-dimensional ICS such that the first $d-1$ terms of the characteristic sequence is given by $(a_1,a_2,\dots,a_{d-1})$.
\end{lemma}

\begin{proof}
    We write $b_i=|p_i-p_{i+1}|$.
    For $1\leq i\leq d+1$, let $R_i$ denote the radius of $S(p_1,p_2,\dots,p_i)$.
    Applying Lemma \ref{key_lemma} (2) to the points $p_1,p_2,\dots,p_i$, we obtain
    $$
        \biggl(\dfrac{R_i}{b_i}\biggr)^2 =
        \dfrac{a_{i-1}}{
            1-\dfrac{a_{i-2}}{
                1-\dfrac{a_{i-3}}{
                    \ddots -\dfrac{a_2}{1-a_1}
                }
            }
        }.
    $$
    Since we have $R_i<b_i\;(1\leq i\leq d)$, we see that (\ref{continued_fraction_ineq}) holds.
    Similarly, $R_{d+1}=b_{d+1}$ shows that (\ref{continued_fraction}) for $i=1$ holds.
    Applying this for the shifted sequence $(p_{i+k})_{i=1}^\infty$, we obtain (\ref{continued_fraction}) for general $i$.

    Conversely, suppose that $(a_i)_{i=1}^\infty$ is a sequence of positive real numbers satisfying (\ref{continued_fraction_ineq}) and (\ref{continued_fraction}).
    We take two points $p_1,p_2\in \mathbb{R}^d$ arbitrarily.
    Suppose that for some $2\leq i \leq d$, we have constructed $p_1,p_2,\dots,p_i\in \mathbb{R}^d$ in good position such that $a_j=b_j^2/4b_{j+1}^2$ holds for $1\leq j\leq i-2$, where $b_j=|p_j-p_{j+1}|$.
    By Lemma \ref{key_lemma} (2), the radius $R$ of $S(p_1,p_2,\dots,p_i)$ is given by
    $$
        R^2 =
                    \dfrac{b_{i-1}^2/4}{
                    1-\dfrac{a_{i-2}}{
                        1-\dfrac{a_{i-3}}{
                            \ddots -\dfrac{a_2}{1-a_1}
                        }
                    }
                }.
    $$
    Let $b_i=b_{i-1}/2\sqrt{a_{i-1}}$.
    Our assumption (\ref{continued_fraction_ineq}) shows that $b_i^2>R^2$, so there exists a point $p_{i+1}\in \mathbb{R}^d$ such that $|p_j-p_{i+1}|=b_i$ holds for $1\leq j\leq i$.
    Repeating this process, we obtain $p_1,p_2,\dots,p_{d+1}\in \mathbb{R}^d$ in good position such that $a_i=b_i^2/4b_{i+1}^2$ for $1\leq i\leq d-1$, where $b_i=|p_i-p_{i+1}|$.
    The claim now follows from Corollary \ref{extension_lemma}.
\end{proof}

\begin{example}\label{periodicity_3_and_4}
    When $d=3$, the recurrence relation (\ref{continued_fraction}) becomes
    $$
        a_{i+2}=1-\dfrac{a_{i+1}}{1-a_i}.
    $$
    Therefore the sequence $(a_i)_{i=1}^\infty$ can be computed in terms of $a_1$ and $a_2$ as follows:
    $$
        a_1,a_2,\dfrac{1-a_1-a_2}{1-a_1},\dfrac{a_1a_2}{(1-a_1)(1-a_2)},\dfrac{1-a_1-a_2}{1-a_2},a_1,a_2,\dots
    $$
    In particular, we see that $a_i=a_{i+5}$ holds for $i\geq 1$.
    This is known as the \emph{Lyness $5$-cycle}.
    When $d=4$, the recurrence relation (\ref{continued_fraction}) becomes
    $$
        a_{i+3}=1-\dfrac{a_{i+2}}{1-\dfrac{a_{i+1}}{1-a_i}}.
    $$
    Therefore the sequence $(a_i)_{i=1}^\infty$ can be computed in terms of $a_1$, $a_2$, and $a_3$ as follows:
    $$
        a_1,a_2,a_3,\dfrac{1-a_1-a_2-a_3+a_1a_3}{1-a_1-a_2},\dfrac{a_1a_2a_3}{(1-a_1-a_2)(1-a_2-a_3)},\dfrac{1-a_1-a_2-a_3+a_1a_3}{1-a_2-a_3},a_1,a_2,a_3,\dots
    $$
    In particular, we see that $a_i=a_{i+6}$ holds for $i\geq 1$.
\end{example}

We want to show that the characteristic sequence of a special $d$-dimensional ICS is $(d+2)$-periodic, i.e., $a_i=a_{i+d+2}$ holds for $i\geq 1$.
This can be checked by a direct computation as in Example \ref{periodicity_3_and_4}, but here we present a more conceptual proof:

\begin{lemma}[Lyness]\label{continued_fraction_periodic}
    Let $K=\mathbb{C}(a_1,a_2,\dots,a_{d-1})$ be a function field, and define $a_d,a_{d+1},\dots\in K$ by (\ref{continued_fraction}).
    Then we have $a_i=a_{i+d+2}$ for $i\geq 1$.
    In particular, the characteristic sequence of a special $d$-dimensional ICS is $(d+2)$-periodic.
\end{lemma}

\begin{proof}
    This is essentially due to \cite{Lyness_cross_ratio}; see also \cite{Griffiths2012}.
    We reproduce the proof for the sake of completeness.
    The idea to use the cross-ratio:
    $$
        C(a,b,c,d)=\dfrac{(a-b)(c-d)}{(a-c)(b-d)}.
    $$
    Let $L=\mathbb{C}(x_1,x_2,\dots,x_{d+2})$ be a function field.
    We regard $K$ as a subfield of $L$ by the embedding
    $$
        K\hookrightarrow L;\quad a_i\mapsto C(x_i,x_{i+1},x_{i+2},x_{i+3})\quad (1\leq i\leq d-1).
    $$
    We will show that $a_d=C(x_d,x_{d+1},x_{d+2},x_1)$.
    Using this repeatedly, we get $a_i=C(x_i,x_{i+1},x_{i+2},x_{i+3})$ for all $i\geq 1$, where we set $x_{i+d+2}=x_i$.
    This implies the desired periodicity.

    For distinct elements $a,b,c\in L$, we write $\varphi[a,b,c]$ for the M\"obius transformation on $\mathbb{P}^1_L$ which sends $a,b,c$ to $0,\infty,1$, respectively.
    Then we have $\varphi[b,c,d](a)=C(a,b,c,d)$.
    We also define a M\"obius transformation $\psi[a,b,c,d]$ of $\mathbb{P}^1_L$ by
    $$
        \psi[a,b,c,d](z)=\dfrac{C(a,b,c,d)}{1-z}.
    $$
    Then $\varphi[b,c,d]^{-1}\circ \psi[a,b,c,d]$ sends $0,\infty,1$ to $a,b,c$, respectively, so we get
    $$
        \varphi[b,c,d]^{-1}\circ \psi[a,b,c,d]=\varphi[a,b,c]^{-1},
    $$
    which implies $\psi[a,b,c,d]=\varphi[b,c,d]\circ\varphi[a,b,c]^{-1}$.
    Therefore we have
    \begin{align*}
    \psi[x_{d},x_{d+1},x_{d+2},x_1]\circ\psi[x_{d-1},x_d,x_{d+1},x_{d+2}]\circ \dots\circ\psi[x_1,x_2,x_3,x_4]
    =\varphi[x_{d+1},x_{d+2},x_1]\circ\varphi[x_1,x_2,x_3]^{-1}.
    \end{align*}
    In particular, the left hand side sends $0$ to $1$, i.e., 
    $$
        \dfrac{C(x_d,x_{d+1},x_{d+2},x_1)}{
            1-\dfrac{C(x_{d-1},x_d,x_{d+1},x_{d+2})}{
                1-\dfrac{C(x_{d-2},x_{d-1},x_d,x_{d+1})}{
                    \ddots -\dfrac{C(x_1,x_2,x_3,x_4)}{1-0}
                }
            }
        }=1.
    $$
    This shows that $a_d=C(x_d,x_{d+1},x_{d+2},x_1)$ as desired.
\end{proof}

\section{Parametrization of characteristic sequences}

In this section we study the parameter space of all possible characteristic sequences of special $d$-dimensional ICSs.
Since such sequences are $(d+2)$-periodic (Lemma \ref{continued_fraction_periodic}), it suffices to consider the first $d+2$ terms.

\begin{definition}
    We define $S_d\subset \mathbb{R}^{d+2}$ to be the set of all sequences $(x_1,x_2,\dots,x_{d+2})$ such that there exists a special $d$-dimensional ICS $p$ with characteristic sequence $(a_i)_{i=1}^\infty$ satisfying $x_i=a_i\;(1\leq i\leq d+2)$.
\end{definition}

\begin{lemma}
    The set $S_d$ has a cyclic symmetry $(x_1,x_2,\dots,x_{d+2})\mapsto (x_2,x_3,\dots,x_{d+2},x_1)$.
    Moreover, we have $S_d\subset (0,1)^{d+2}$.
\end{lemma}

\begin{proof}
    The first claim is clear from the periodicity of characteristic sequences.
    By Lemma \ref{recurrence}, we have $x_1\in (0,1)$ on $S_d$ .
    This implies $S_d\subset (0,1)^{d+2}$ by the cyclic symmetry.
\end{proof}

In order to write down the recurrence relation (\ref{recurrence}) without using continued fractions, we define the following auxiliary polynomials.

\begin{definition}\label{Kanda_function}
    For $n\geq 1$, we define a polynomial $F^{(n)}(x_1,x_2,\dots,x_n)$ by
    $$
        F^{(n)}(x_1,x_2,\dots,x_n)=\sum_{\substack{A\subset \{1,2,\dots,n\}\\ \forall i,j\in A,\:|i-j|\geq 2}} (-1)^{|A|}x_A,
    $$
    where $x_A=\prod_{i\in A}x_i$.
    For example, we have $F^{(1)}(x_1)=1-x_1$, $F^{(2)}(x_1,x_2)=1-x_1-x_2$, and $F^{(3)}(x_1,x_2,x_3)=1-x_1-x_2-x_3+x_1x_3$.
    We set $F^{(i)}=1$ for $i\leq 0$.
\end{definition}

\begin{lemma}\label{Kanda_function_properties}
    The following formulas hold true:
    \begin{enumerate}
        \item $F^{(n)}(x_1,x_2,\dots,x_n)=F^{(n-1)}(x_1,x_2,\dots,x_{n-1})-x_nF^{(n-2)}(x_1,x_2,\dots,x_{n-2})$.
        \item $F^{(n-1)}(x_1,\dots,x_{n-1})F^{(n-1)}(x_2,\dots,x_n)-F^{(n)}(x_1,\dots,x_n)F^{(n-2)}(x_2,\dots,x_{n-2})=x_1x_2\cdots x_n$.
        \item $1-\dfrac{x_n}{1-\dfrac{x_{n-1}}{1-\dfrac{x_{n-2}}{\ddots-\dfrac{x_2}{1-x_1}}}} = \dfrac{F^{(n)}(x_1,x_2,\dots,x_n)}{F^{(n-1)}(x_1,x_2,\dots,x_{n-1})}$.
    \end{enumerate}
\end{lemma}

\begin{proof}
    The formula (1) is clear from the definition.
    Let us prove (2) by induction on $n$.
    The claim is trivial for $n=1$, so we assume that $n\geq 2$.
    By (1), we have
    \begin{align*}
        &F^{(n-1)}(x_1,\dots,x_{n-1})F^{(n-1)}(x_2,\dots,x_n)\\
        ={}&F^{(n-1)}(x_1,\dots,x_{n-1})(F^{(n-2)}(x_2,\dots,x_{n-1})-x_nF^{(n-3)}(x_2,\dots,x_{n-2})),\\
        &F^{(n)}(x_1,\dots,x_n)F^{(n-2)}(x_2,\dots,x_{n-1})\\
        ={}&(F^{(n-1)}(x_1,\dots,x_{n-1})-x_nF^{(n-2)}(x_1,\dots,x_{n-2}))F^{(n-2)}(x_2,\dots,x_{n-1}).
    \end{align*}
    Therefore the left hand side of (2) is equal to
    $$
        x_n(F^{(n-1)}(x_1,\dots,x_{n-1})F^{(n-3)}(x_2,\dots,x_{n-2})-F^{(n-2)}(x_1,\dots,x_{n-2})F^{(n-2)}(x_2,\dots,x_{n-1})),
    $$
    which is, by the induction hypothesis, equal to the right hand side of (2).
    Finally, we prove (3) by induction on $n$.
    The claim is trivial for $n=1$, so we assume that $n\geq 2$.
    By the induction hypothesis, the left hand side is equal to
    $$
    1-\dfrac{x_n}{\biggl(\dfrac{F^{(n-1)}(x_1,x_2,\dots,x_{n-1})}{F^{(n-2)}(x_1,x_2,\dots,x_{n-2})}\biggr)}
    =\dfrac{F^{(n-1)}(x_1,x_2,\dots,x_{n-1})-x_nF^{(n-2)}(x_1,x_2,\dots,x_{n-2})}{F^{(n-1)}(x_1,x_2,\dots,x_{n-1})}.
    $$
    The desired formula follows from this and (1).
\end{proof}

\begin{lemma}\label{recurrence_explicit}
    Let $p$ be a special $d$-dimensional ICS.
    Then the characteristic sequence $(a_i)_{i=1}^\infty$ satisfies
    \begin{align}
        \label{continued_fraction_ineq_x}
        F^{(i)}(a_1,a_2,\dots,a_i) &>0\quad (1\leq i\leq d-1),\\
        \label{continued_fraction_x}
        F^{(d)}(a_i,a_{i+1},\dots,a_{i+d-1})&=0\quad (i\geq 1).
    \end{align}
    Conversely, for any sequence of positive real numbers $(a_1,a_2,\dots,a_{d-1})$ satisfying (\ref{continued_fraction_ineq_x}), there is a special $d$-dimensional ICS such that the first $d-1$ terms of the characteristic sequence is given by $(a_1,a_2,\dots,a_{d-1})$.
\end{lemma}

\begin{proof}
    This follows from Lemmas \ref{recurrence} and \ref{Kanda_function_properties}.
\end{proof}

\begin{definition}
    We write $F[i,j]=F^{(j-i+1)}(x_i,x_{i+1},\dots,x_j)$.
    We define $U_d\subset \mathbb{R}^{d-1}$ to be the open subset of $\mathbb{R}^{d-1}$ defined by
    $$
        U_d = \{(x_1,x_2,\dots,x_{d-1})\in \mathbb{R}^{d-1}\mid x_i>0,\;F[1,i]>0\;(1\leq i\leq d-1)\}.
    $$
\end{definition}

\begin{theorem}
    There is a bijection
    $
    S_d\xrightarrow{\sim} U_d;\;(x_1,x_2,\dots,x_{d+2})\mapsto (x_1,x_2,\dots,x_{d-1}).
    $
\end{theorem}

\begin{proof}
    For any $(x_1,x_2,\dots,x_{d+2})$, we have $(x_1,x_2,\dots,x_{d-1})\in U_d$ by Lemma \ref{recurrence}.
    Moreover, the same lemma shows that $S_d\to U_d$ is surjective.
    Since $x_d,x_{d+1},x_{d+2}$ are uniquely determined by $x_1,x_2,\dots,x_{d-1}$ by the recurrence relation (\ref{continued_fraction}), this map is bijective.
\end{proof}

The above theorem says that the parameter space $S_d$ of characteristic sequences can be identified with the open subset $U_d$ of $\mathbb{R}^{d-1}$.
In particular, we can regard $x_d,x_{d+1},x_{d+2}$ as functions on $U_d$.
Next we study the function $\sqrt{x_1x_2\cdots x_{d+2}}$ on $S_d$ using this identification.

\begin{lemma}\label{product_max_exist}
    The function $\sqrt{x_1x_2\cdots x_{d+2}}$ on $S_d$ attains a maximum, and its infimum is $0$.
\end{lemma}

\begin{proof}
    For $\varepsilon\in (0,1)$, we define a subset $V_\varepsilon\subset U_d$ by
    $$
        V_\varepsilon = \{(x_1,x_2,\dots,x_{d-1})\in \mathbb{R}^{d-1}\mid x_i\geq \varepsilon,\; F[1,i]\geq \varepsilon^i\;(1\leq i\leq d-1)\}.
    $$
    Then $V_\varepsilon$ is compact and we have $U_d=\bigcup_{\varepsilon\in (0,1)}V_\varepsilon$.
    Therefore it suffices to show that $x_1x_2\cdots x_{d+2}<\varepsilon$ holds on $U_d\setminus V_\varepsilon$.
    Since $S_d\subset (0,1)^{d+2}$, it suffices to show that for any $(x_1,x_2,\cdots,x_{d-1})\in U_d\setminus V_\varepsilon$, there is some $i$ with $1\leq i\leq d+2$ such that $x_i<\varepsilon$.

    Suppose towards contradiction that $(x_1,x_2,\cdots,x_{d-1})\in U_d\setminus V_\varepsilon$ and $x_i\geq \varepsilon$ for all $1\leq i\leq d+2$.
    Then there is some $i$ with $1\leq i\leq d-1$ such that $F[1,i]<\varepsilon^i$.
    Let $i_0$ be the minimum of such $i$.
    Then we have $F[1,i_0]\geq \varepsilon^{i_0-1}$, $F[1,i_0]<\varepsilon^{i_0}$ and hence
    $$
        F[1,i_0+1]=F[1,i_0]-x_{i_0+1}F[1,i_0-1]<\varepsilon^{i_0}-\varepsilon^{i_0}=0.
    $$
    This contradicts to the fact that $F[1,i]>0\;(1\leq i\leq d-1)$ and $F[1,d]=0$ on $U_d$.
\end{proof}

In order to compute the maximum value of $\sqrt{x_1x_2\cdots x_{d+2}}$ on $S_d$, we write down the functions $x_d,x_{d+1},x_{d+2}$ on $U_d$ in terms of $x_1,x_2,\dots,x_{d-1}$.

\begin{lemma}\label{Kanda_triple}
    As functions on $U_d$, we have
    $$
        x_d = \dfrac{F[1,d-1]}{F[1,d-2]},\quad
        x_{d+1} = \dfrac{x_1x_2\cdots x_{d-1}}{F[1,d-2]F[2,d-1]},\quad
        x_{d+2} = \dfrac{F[1,d-1]}{F[2,d-1]}.
    $$
\end{lemma}

\begin{proof}
    By the recurrence relation (\ref{continued_fraction_x}), we have
    $0 = F[1,d] = F[1,d-1] - x_dF[1,d-2]$
    on $U_d$, which proves the first formula.
    As for the second formula, we compute as follows:
    \begin{align*}
        x_{d+1} &= \dfrac{F[2,d]}{F[2,d-1]}=\dfrac{F[2,d-1]-x_dF[2,d-2]}{F[2,d-1]}\\
        &=\dfrac{F[1,d-2]F[2,d-1]-x_dF[1,d-2]F[2,d-2]}{F[1,d-2]F[2,d-1]}\\
        &=\dfrac{F[1,d-2]F[2,d-1]-F[1,d-1]F[2,d-2]}{F[1,d-2]F[2,d-1]}\\
        &=\dfrac{x_1x_2\cdots x_{d-1}}{F[1,d-2]F[2,d-1]}.
    \end{align*}
    Here, the last equality follows from Lemma \ref{Kanda_function_properties}.
    The third formula can be proved similarly.
\end{proof}

\begin{corollary}\label{product_explicit}
    As a function on $U_d$, we have $\sqrt{x_1x_2\cdots x_{d+2}} = \dfrac{x_1x_2\cdots x_{d-1}F[1,d-1]}{F[1,d-2]F[2,d-1]}$.
\end{corollary}

\begin{lemma}\label{partial_derivative}
    Regarding $\sqrt{x_1x_2\cdots x_{d+2}}$ as a function on $U_d$, we have
    $$
        \dfrac{\partial}{\partial x_{d-1}}\sqrt{x_1x_2\cdots x_{d+2}}=0\iff x_d=x_{d+1}.
    $$
\end{lemma}

\begin{proof}
    The left hand side is equivalent to $\dfrac{\partial}{\partial x_{d-1}}\log (\sqrt{x_1x_2\cdots x_{d+2}})=0$.
    Note that we have
    $$
        \dfrac{\partial}{\partial x_j} F[i,j]
        = \dfrac{\partial}{\partial x_j} (F[i,j-1]-x_j F[i,j-2]) = -F[i,j-2].
    $$
    Using this formula and Corollary \ref{product_explicit}, we can compute the logarithmic derivative as follows:
    \begin{align*}
        \dfrac{\partial}{\partial x_{d-1}}\log (\sqrt{x_1x_2\cdots x_{d+2}})
        & = \dfrac{1}{x_{d-1}}-\dfrac{F[1,d-3]}{F[1,d-1]}+\dfrac{F[2,d-3]}{F[2,d-1]}\\
        & = \dfrac{1}{x_{d-1}}-\dfrac{F[1,d-3]F[2,d-1]-F[1,d-1]F[2,d-3]}{F[1,d-1]F[2,d-1]}.
    \end{align*}
    On the other hand, Lemma \ref{Kanda_function_properties} shows that
    \begin{align*}
        &F[1,d-3]F[2,d-1]-F[1,d-1]F[2,d-3]\\
        =&F[1,d-3](F[2,d-2]-x_{d-1}F[2,d-3]) - (F[1,d-2]-x_{d-1}F[1,d-3])F[2,d-3]\\
        =&F[1,d-3]F[2,d-2]-F[1,d-2]F[2,d-3]\\
        =&x_1x_2\cdots x_{d-2}.
    \end{align*}
    Combining these two results, we see that
    \begin{align*}
        \dfrac{\partial}{\partial x_{d-1}}\sqrt{x_1x_2\cdots x_{d+2}}=0
        &\iff \dfrac{1}{x_{d-1}}=\dfrac{x_1x_2\cdots x_{d-2}}{F[1,d-1]F[2,d-1]}\\
        &\iff x_1x_2\cdots x_{d-1} = F[1,d-1]F[2,d-1].
    \end{align*}
    By Lemma \ref{Kanda_triple}, this is equivalent to $x_d=x_{d+1}$.
\end{proof}

\begin{lemma}\label{product_max}
    $\max_{x\in S}\sqrt{x_1x_2\cdots x_{d+2}}=(2^{d+2}\cos^{d+2}(\pi/(d+2)))^{-1}$.
\end{lemma}

\begin{proof}
    By Lemma \ref{product_max_exist}, we can take a point $(a_1,a_2,\cdots,a_{d+2})\in S$ which attains the maximum value of $\sqrt{x_1x_2\cdots x_{d+2}}$.
    Then by Lemma \ref{partial_derivative}, we have $a_{d-1}=a_d$.
    Since $S_d$ has a cyclic symmetry, we get $a_1=a_2=\cdots=a_d=t$ for some $t>0$.
    In order to determine the value of $t$, we use that fact that
    $$
        F^{(i)}(t,t,\dots,t)>0\quad(1\leq i\leq d-1),\quad F^{(d)}(t,t,\cdots,t)=0.
    $$
    Set $u_n=F^{(n)}(t,t,\cdots,t)$.
    By definition, we have $u_1=1-t$, $u_2=1-2t$, and $u_n=u_{n-1}-tu_{n-2}$.
    If $t\leq 1/4$, then we can easily see that $u_n>0$ for all $n$, which contradicts to $u_d=0$.
    Therefore we have $t>1/4$.
    Solving the recurrence relation, we get
    $$
        u_n=\dfrac{\alpha^{n+2}-\beta^{n+2}}{\alpha-\beta},
    $$
    where $\alpha=(1+\sqrt{1-4t})/2$ and $\beta=(1-\sqrt{1-4t})/2$ are two solutions to $X^2-X+t=0$.
    By $u_i>0\;(1\leq i\leq d-1)$ and $u_d=0$, we see that the imaginary part of $\alpha^i$ is positive for $i=3,4,\cdots,d+1$ and zero for $i=d+2$.
    This implies $\arg \alpha = \pi/(d+2)$, $|\alpha|=(2\cos(\pi/(d+2)))^{-1}$, and $t=|\alpha|^2$.
    By Lemma \ref{product_explicit}, the value of $\sqrt{x_1x_2\cdots x_{d+2}}$ at $(t,t,\cdots,t)$ is given by
    $$
        t^{d-1}\cdot \dfrac{u_{d-1}}{u_{d-2}^2}=|\alpha|^{2d-2}\cdot \dfrac{(\alpha^{d+1}-\beta^{d+1})(\alpha-\beta)}{(\alpha^d-\beta^d)^2}=|\alpha|^d\cdot \dfrac{\sin^2(\pi/(d+2))}{\sin^2(2\pi/(d+2))},
    $$
    which is equal to $(2^{d+2}\cos^{d+2}(\pi/(d+2)))^{-1}$.
    This completes the proof.
\end{proof}

\section{Structure of ICSs}

In this section we study the structure of ICSs in terms of their characteristic sequence.

\begin{definition}
    Let $p$ be a special $d$-dimensional ICS with characteristic sequence $(a_i)_{i=1}^\infty$.
    We define the \emph{scale factor} of $p$ by
    $$
        r=\dfrac{1}{2^{d+2}\sqrt{a_1a_2\cdots a_{d+2}}}.
    $$
\end{definition}

\begin{lemma}\label{scale_factor_existence}
    Let $p$ be a special $d$-dimensional ICS with scale factor $r$.
    Then we have
    $$
        |p_{i+d+2}-p_{i+d+3}|=r|p_i-p_{i+1}|\quad (i\geq 1).
    $$
\end{lemma}

\begin{proof}
    Let $(a_i)_{i=1}^\infty$ be the characteristic sequence of $p$.
    By Lemma \ref{recurrence} and Lemma \ref{continued_fraction_periodic}, we have $a_i=a_{i+d+2}$ for $i\geq 1$.
    In particular, we have
    $$
        \dfrac{|p_{i+d+2}-p_{i+d+3}|}{|p_i-p_{i+1}|}
        = \prod_{j=i}^{i+d+1}\dfrac{|p_{j+1}-p_{j+2}|}{|p_j-p_{j+1}|}
        = \dfrac{1}{2^{d+2}\sqrt{a_ia_{i+1}\cdots a_{i+d+1}}}=r
    $$
    for any $i\geq 1$.
    This proves the claim.
\end{proof}

\begin{lemma}\label{parallel_lemma}
    Let $p$ be a special $d$-dimensional ICS with scale factor $r$.
    Then we have
    $$
        p_{i+d+2}-p_{i+d+3} = -r (p_i-p_{i+1}) \quad (i\geq 1).
    $$
\end{lemma}

\begin{proof}
    It suffices to prove that this holds for $i=1$.
    Let $H=H(p_3,p_4,\dots,p_{d+2})$ and let $Q$ be the center of $S(p_3,p_4,\dots,p_{d+2})$.
    By definition of ICS, we have
    $$
        |p_1-p_j|=|p_2-p_j|\quad (3\leq j\leq d+2).
    $$
    Therefore $H$ coincides with the hyperplane $\{x\in \mathbb{R}^d\mid |p_1-x|=|p_2-x|\}$, which is perpendicular to the line $p_1p_2$.
    On the other hand, applying Lemma \ref{key_lemma} (1) to the points $p_3,p_4,\dots,p_{d+3}$, we see that the half-line $p_{d+3}Q$ is also perpendicular to $H$, and the point $p_{d+4}$ lies on this half-line.
    Therefore the lines $p_ip_{i+1}$ and $p_{d+3}p_{d+4}$ are parallel.

    Applying Lemma \ref{key_lemma} (3) to the points $p_2,p_3,\dots,p_{d+2}$, we see that the segment $p_2p_{d+3}$ is disjoint from $H$.
    In other words, the points $p_2$ and $p_{d+3}$ are on the same side with respect to the hyperplane $H$.
    Combining with the fact that $p_{d+4}$ lies on the half-line $p_{d+3}Q$, we conclude that the vectors $p_1-p_2$ and $p_{d+3}-p_{d+4}$ are in the opposite direction.
    This completes the proof.
\end{proof}

\begin{theorem}\label{shift_vector_existence}
    Let $p$ be a special $d$-dimensional ICS with scale factor $r$.
    Then there exists a point $v \in \mathbb{R}^d$ such that $p_{i+d+2} =  v-r p_i$ holds for $i\geq 1$.
    In particular, $p$ is periodic if and only if $r=1$.
\end{theorem}

\begin{proof}
    Define $v$ to be the unique vector satisfying $p_{d+3}=v-rp_1$.
    Since we have $p_{i+d+2}-p_{i+d+3}= -r(p_i-p_{i+1})$ for $i\geq 1$ by Lemma \ref{parallel_lemma}, we get $p_{i+d+2}=v-r p_i$ for $i\geq 1$ inductively.
\end{proof}

\begin{corollary}
    For any $r\in [\cos^{d+2}(\pi/(d+2)),\infty)$, there is a special $d$-dimensional ICS with scale factor $r$.
    In particular, there exists a periodic $d$-dimensional ICS for any $d\geq 2$.
\end{corollary}

\begin{proof}
    This follows from Lemma \ref{product_max} and Theorem \ref{shift_vector_existence}.
\end{proof}

\begin{example}
    We illustrate our results in the case $d=3$.
    By Lemma \ref{continued_fraction_periodic}, the characteristic sequence of a special $3$-dimensional ICS is $5$-periodic.
    The parameter space $U_3$ of characteristic sequences is given by
    $$
        U_3 = \{(x_1,x_2)\in \mathbb{R}^2\mid x_1>0,\;x_2>0,\;x_1+x_2< 1\}.
    $$
    By Lemma \ref{Kanda_triple}, the functions $x_3,x_4,x_5$ on $U_3$ are given by
    $$
        x_3=\dfrac{1-x_1-x_2}{1-x_1},\quad x_4=\dfrac{x_1x_2}{(1-x_1)(1-x_2)},\quad x_5=\dfrac{1-x_1-x_2}{1-x_2}.
    $$
    Hence the function $\sqrt{x_1x_2x_3x_4x_5}$ can be written as
    $$
        \sqrt{x_1x_2x_3x_4x_5} = \dfrac{x_1x_2(1-x_1-x_2)}{(1-x_1)(1-x_2)}.
    $$
    By Lemma \ref{product_max}, this function attains the maximum value $(2^5\cos^5(\pi/5))^{-1}=0.0901\dots$ at $(t,t)$, where $t=(4\cos^2(\pi/5))^{-1}=0.3819\dots$.
    Therefore the scale factor $r$ can take any value in $[\cos^5(\pi/5),\infty)$, where $\cos^5(\pi/5)=0.3465\dots$.
    Periodic ICSs correspond to points with $\sqrt{x_1x_2x_3x_4x_5} = 1/32$, which forms a simple closed curve as in the figure below.
\end{example}
\begin{figure}[h]
\includegraphics[width=3.5cm]{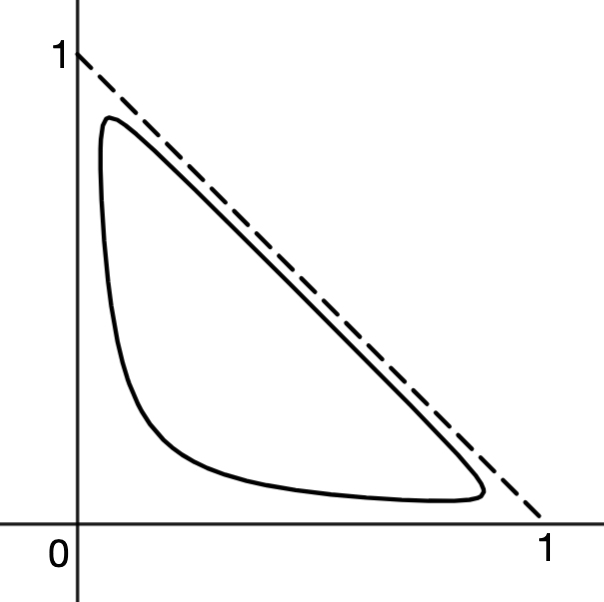}
\caption{Parameters of periodic $3$-dimensional ICSs}
\end{figure}

For example, the following figure shows the $3$-dimensional periodic ICS with period $10$ corresponding to the parameter $(x_1,x_2)=\biggl(\dfrac{1}{16}, \dfrac{45+\sqrt{105}}{64}\biggr)$ from two perspectives.

\begin{figure}[h]
\includegraphics[height=5cm]{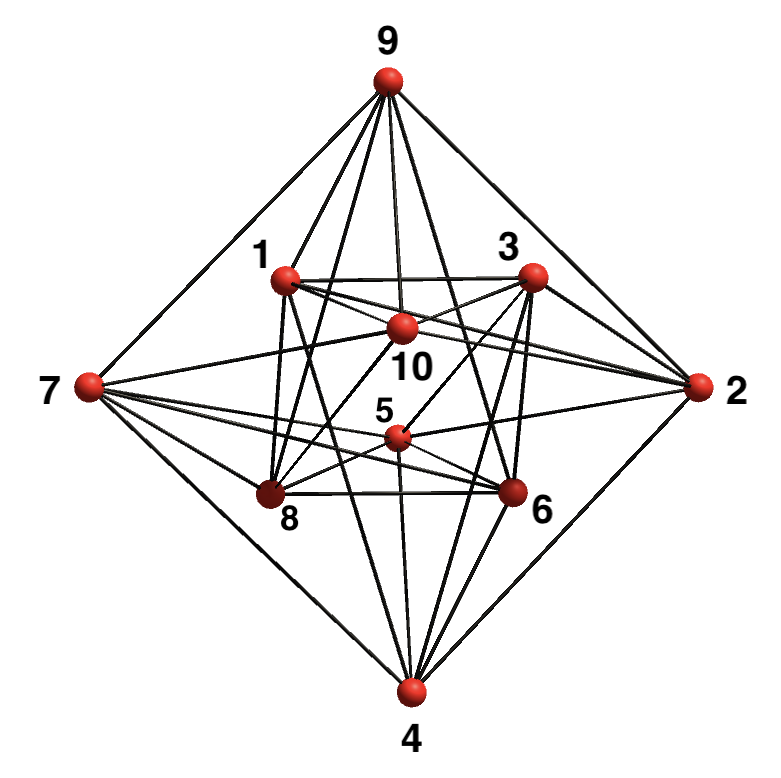}
\includegraphics[height=5cm]{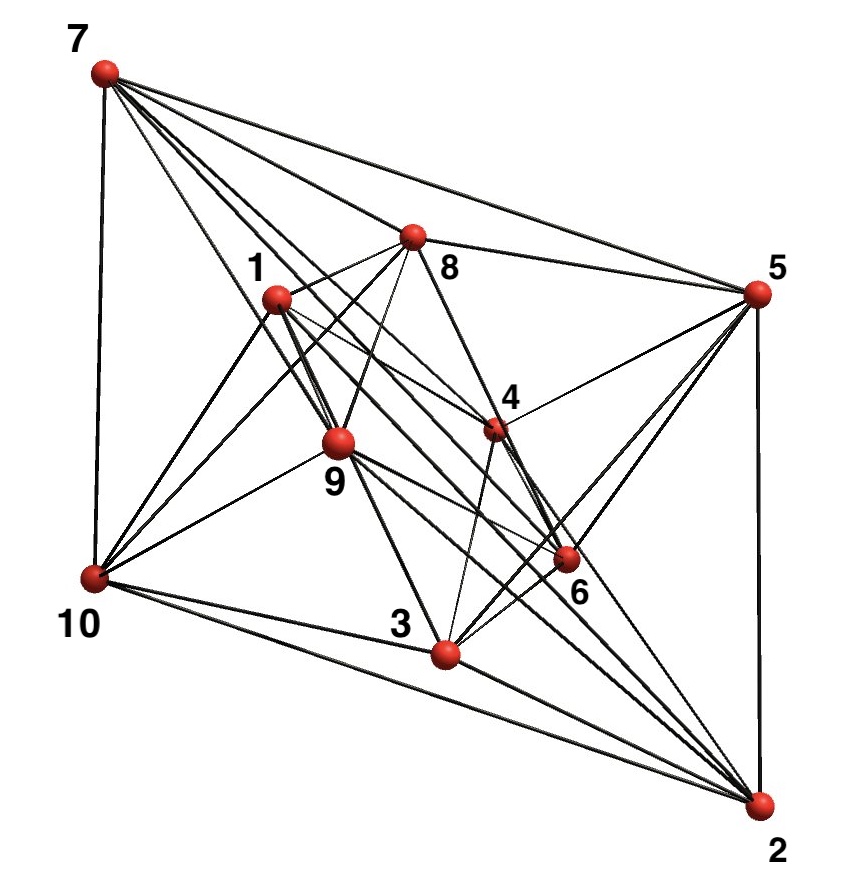}
\caption{A periodic $3$-dimensional ICS}
\end{figure}

Finally, we prove Goddyn's conjecture on periodic ICSs.

\begin{theorem}[Goddyn's conjecture]
    Let $p$ be a $d$-dimensional ICS.
    If $p$ is periodic, then its period must be $2d+4$.
\end{theorem}

\begin{proof}
    First we note that $p$ is special since it is periodic.
    Let $m$ be the period of $p$.
    By definition of ICS, the points $p_1,p_2,\dots,p_{d+1}$ are distinct, so we have $m>d+1$.
    Suppose that $m=d+2$.
    For each $i\in \{1,2,\dots,d+2\}$, the point $p_{i+d+2}$ is the center of $S(p_{i+1},\dots,p_{i+d+1})$.
    By the periodicity, this implies that $p_i$ is the center of $S(p_1,\dots,p_{i-1},p_{i+1},\dots,p_{d+2})$.
    Varying $i$, we conclude that $|p_i-p_j|$ is constant for $1\leq i<j\leq d+2$, which is impossible in $\mathbb{R}^d$.
    Therefore we have $m>d+2$.

    On the other hand, the scale factor of $p$ must be $1$ by the periodicity.
    By Lemma \ref{shift_vector_existence}, there is some vector $v\in \mathbb{R}^d$ such that $p_{i+d+2}=v-p_i$ holds for $i\geq 1$.
    In particular, we have
    $$
        p_{i+2d+4} = v-(v-p_i) = p_i\quad (i\geq 1),
    $$
    so we must have $m\mid 2d+4$.
    Therefore the only possibility is $m=2d+4$.
\end{proof}

\printbibliography

\end{document}